\documentclass[12pt]{article}
\usepackage{graphicx}
\usepackage{amsmath,amssymb,amsthm,amsfonts}
\newtheorem{thm}{Theorem}[section]
\newtheorem{cor}[thm]{Corollary}
\newtheorem{lem}[thm]{Lemma}
\newtheorem{prop}[thm]{Proposition}

\newcommand{\R}{{\mathbb{R}}}
\newcommand{\Z}{{\mathbb{Z}}}

\newcommand{\1}{\partial}
\newcommand{\2}{\overline}
\newcommand{\3}{\varepsilon}
\newcommand{\4}{\widetilde}

\def\ni{\noindent}

\begin{document}
\title{Collapsing behaviour of a singular diffusion 
equation}
\author{Kin Ming Hui\\
%\thanks{ }\\
Institute of Mathematics, Academia Sinica,\\
Nankang, Taipei, 11529, Taiwan, R. O. C.}
\date{May 27, 2011}
\smallbreak \maketitle
\begin{abstract}
Let $0\le u_0(x)\in L^1(\R^2)\cap L^{\infty}(\R^2)$ be such that $u_0(x)
=u_0(|x|)$ for all $|x|\ge r_1$ and is monotone decreasing 
for all $|x|\ge r_1$ for some constant $r_1>0$ and 
$\mbox{ess}\inf_{\2{B}_{r_1}(0)}u_0\ge\mbox{ess}
\sup_{\R^2\setminus B_{r_2}(0)}u_0$ for some constant 
$r_2>r_1$. Then under some mild decay conditions at infinity on
the initial value $u_0$ we will extend the result of P.~Daskalopoulos,
M.A.~del Pino and N.~Sesum \cite{DP2}, \cite{DS2}, and prove the 
collapsing behaviour of the maximal solution of the 
equation $u_t=\Delta\log u$ in $\R^2\times (0,T)$, $u(x,0)=u_0(x)$ 
in $\R^2$, near its extinction time $T=\int_{R^2}u_0dx/4\pi$
without using the Hamilton-Yau Harnack inequality.
\end{abstract}

\vskip 0.2truein


\vskip 0.2truein
\setcounter{equation}{0}
\setcounter{section}{-1}

\section{Introduction}
\setcounter{equation}{0}
\setcounter{thm}{0}

Recently there is a lot of study on the equation,
\begin{equation}
\left\{\begin{aligned}
&u_t=\Delta\log u, u>0,\quad\mbox{ in }\R^2\times (0,T)\\
&u(x,0)=u_0\qquad\qquad\,\,\mbox{ in }\R^2
\end{aligned}\right.
\end{equation}
by P.~Daskalopoulos, M.A.~del Pino and N.~Sesum \cite{DP1}, \cite{DP2}, 
\cite{DS1}, \cite{DS2}, S.Y.~Hsu \cite{Hs1}, \cite{Hs2}, \cite{Hs3}, \cite{Hs4}, J.R.~Esteban, A.~Rodriguez and J.L.~Vazquez \cite{ERV1}, \cite{ERV2}, K.M.~Hui \cite{Hu1}, \cite{Hu2}, and L.F.~Wu \cite{W1}, 
\cite{W2}, etc. Equation (0.1) 
arises in the model of thin films of fluid speading on a solid surface 
\cite{G} when the fourth order term is neglected. It also arises in the 
study of the Ricci flow on surfaces \cite{W1}, \cite{W2}, and as the 
singular limit of the porous medium equation, 
\begin{equation*}
\left\{\begin{aligned}
&u_t=\Delta\left(\frac{u^m}{m}\right), u>0,\quad\mbox{ in }\R^2\times (0,T)\\
&u(x,0)=u_0(x)\qquad\qquad\,\,\,\mbox{in }\R^2
\end{aligned}\right.
\end{equation*}
as $m\to 0$ \cite{Hu2}, \cite{ERV2}. In \cite{DP1} and \cite{Hu1}
P.~Daskalopoulos and M.A.~del Pino and K.M.~Hui proved that
corresponding to each $0\le u_0\in L^p(\R^2)\cap L^1(\R^2)$, $p>1$, 
and $2\le f\in L^1(0,\infty)$ there exists a classical solution $u$ of (0.1) 
in $\R^n\times (0,T)$ satisfying the mass loss equation,
\begin{equation}
\int_{\R^2}u(x,t)\,dx=\int_{\R^2}u_0\,dx-2\pi\int_0^tf(s)\,ds\quad\forall
0\le t<T
\end{equation}
where $T>0$ given by
\begin{equation}
\int_{\R^2}u_0\,dx=2\pi\int_0^Tf(s)\,ds
\end{equation}
is the extinction time for the solution $u$. Note that the maximal solution
of (0.1) is the solution of (0.1) that satisfies (0.2) with $f\equiv 2$.
When $f\equiv\gamma>2$ is a 
constant, the asymptotic behaviour of the solution $u$ is studied
by S.Y.~Hsu in \cite{Hs3} and \cite{Hs4}. In \cite{Hs3} S.Y.~Hsu proved 
that if $u_0$ is radially symmetric and monotone decreasing and $u$ is the 
solution of (0.1) in $\R^2\times (0,T)$ that satisfies (0.2), (0.3), with
$f\equiv 4$, then the rescaled solution
\begin{equation*}
v(x,s)=\frac{u(x,t)}{T-t},\quad, s=-\log (T-t)
\end{equation*}
will converge uniformly on every compact subsets of $\R^2$ to the function
$$
\frac{8\lambda}{(\lambda +|x|^2)^2}
$$ 
as $s\to\infty$ for some constant $\lambda>0$. For the general case 
$f\equiv\gamma>2$ S.Y.~Hsu \cite{Hs4} proved 
that if $u_0$ is radially symmetric and 
monotone decreasing and $u$ is the solution of (0.1) in $\R^2\times (0,T)$ 
that satisfies (0.2) and (0.3) with $f\equiv\gamma$
and
\begin{equation*}
\lim_{r\to\infty}\frac{ru_r(r,t)}{u(r,t)}=-\gamma\quad\mbox{ uniformly 
on $[a,b]$}\quad\forall 0<a<b<T
\end{equation*}
where $T$ is given by (0.3) with $f\equiv\gamma$,
then there exist unique constants $\alpha>0$, $\beta>-1/2$, 
$\alpha =2\beta +1$, depending only on $\gamma$ such that the 
rescaled function
$$
v(y,s)=\frac{u(y/(T-t)^{\beta},t)}{(T-t)^{\alpha}}
$$
where
\begin{equation*}
s=-\text{log }(T-t)
\end{equation*}
will converge uniformly on every compact subset of $\R^2$ to
$\phi_{\lambda ,\beta}(y)$ for some constant $\lambda>0$ as 
$s\to\infty$ where $\phi_{\lambda ,\beta}(y)
=\phi_{\lambda ,\beta}(|y|)$ is radially symmetric and satisfies
the following O.D.E.
\begin{equation*}
\frac{1}{r}\biggl (\frac{r\phi '}{\phi}\biggr )'+\alpha\phi 
+\beta r\phi '=0\quad\text{ in }(0,\infty)
\end{equation*}
with
$$
\phi (0)=1/\lambda,\phi'(0)=0.
$$
However as observed by J.R.~King \cite{K} using formal asymptotic
analysis as $t$ approaches $T$ the vanishing behaviour for the finite 
mass solution 
of (0.1) which corresponds to the case $f\equiv 2$ is very different
from the vanishing behaviour for the case $f\equiv\gamma>2$. 
J.R.~King find that for compactly supportly finite mass initial 
value the maximal solution of (0.1) behaves like
\begin{equation*}
\frac{(T-t)^2}{\frac{T}{2}|x|^2+e^{\frac{2T}{(T-t)}}}
\end{equation*}
in the inner region $(T-t)\log |x|\le T$ and behaves like 
\begin{equation*}
\frac{2t}{|x|^2(\log |x|)^2}
\end{equation*}
in the outer region $(T-t)\log |x|\ge T$ as $t\nearrow T$. In \cite{DP2}
P.~Daskalopoulos and M.A.~del Pino give a rigorous proof of an extension
of this formal result for radially symmetric initial value $u_0(r)$ 
satisfying the conditions,
\begin{equation}
\mbox{$u_0(x)=u_0(|x|)$ is decreasing on $r=|x|\ge r_1$ for 
some constant $r_1>0$}
\end{equation}
\begin{equation}
u_0(x)=\frac{2\mu}{|x|^2(\log |x|)^2}(1+o(1))\quad\mbox{ as }
|x|\to\infty,
\end{equation}
for some constant $\mu>0$ and 
\begin{equation}
R_0(x):=-\frac{\Delta\log u_0}{u_0}\ge -\frac{1}{\mu}\quad\mbox{ on }
\R^2.
\end{equation}
Note that (0.1) is equivalent to the Ricci flow equation (\cite{W2}) 
$$
\frac{\1}{\1 t}g_{ij}=-2R_{ij}\quad\mbox{ in }\R^2\times (0,T)
$$
where $g_{ij}(t)=g_{ij}(x,t)=u(x,t)\delta_{ij}$ is an evolving metric on 
$\R^2$ and $R_{ij}(x,t)$ is the Ricci curvature of the metric $g_{ij}(t)$.
Then the scalar curvature $R(x,t)$ with respect to the metric 
$g_{ij}(x,t)$ is given by
$$
R(x,t)=-\frac{\Delta\log u}{u}.
$$ 
 Thus (0.6) says that the initial scalar curvature 
is greater than or equal to $-1/\mu$ on $\R^2$.

In \cite{DS2} P.~Daskalopoulos and N.~Sesum extend the result  
to the case of compactly supported $0\le u_0\in L^1(\R^2)\cap 
L^{\infty}(\R^2)$. However their proof of the behaviour of the maximal 
solution in the outer region near the extinction time is very difficult 
and uses the Hamilton-Yau Harnack inequality \cite{HY}. 
    
In this paper we will extend their result to the case of initial value
$0\le u_0\in L^1(\R^2)\cap L^{\infty}(\R^2)$ that satisfies 
(0.4), (0.5), and (0.6) for some constant $\mu\ge 0$ with
the right hand side being replaced by $-\infty$ if $\mu=0$ and 
\begin{equation}
\mbox{ess}\inf_{\2{B}_{r_1}(0)}u_0
\ge\mbox{ess}\sup_{\R^2\setminus B_{r_2}(0)}u_0
\end{equation}
for some constant $r_2>r_1$. Note that (0.7) is automatically satisfied
if $u_0$ has compact support in $\R^2$. We will prove the behaviour 
of the maximal 
solution in the outer region near the extinction time by elementary
method without using the difficult Hamilton-Yau Harnack inequality 
for surfaces \cite{HY}. We will also prove the behaviour of the 
maximal solution in the inner region as the extinction time is approached.

In this paper we will assume that the initial value $u_0$ satisfies
$0\le u_0\in L^1(\R^2)\cap L^{\infty}(\R^2)$, (0.4), (0.5), (0.6) and 
(0.7) for some constants $r_2>r_1>0$ and $\mu\ge 0$ 
with the right hand side being replaced by $-\infty$ if $\mu=0$. 
We will assume that $u$ is the maximal solution of (0.1) in 
$\R^2\times (0,T)$ that satisfies (0.2) with $f\equiv 2$ and
$$
T=\frac{1}{4\pi}\int_{\R^2}u_0\,dx
$$
for the rest of the paper. For any $x_0\in\R^2$ and $r_0>0$ let
$B_{r_0}(x_0)=\{x\in\R^2:|x-x_0|<r_0\}$ and $B_{r_0}=B_{r_0}(0)$.
Note that by \cite{Hu1} $u$ satisfies the  
Aronson-B\'enilan inequality,
\begin{equation}
\frac{u_t}{u}\le\frac{1}{t}\quad\mbox{ in }\R^2\times (0,T).
\end{equation}

The plan of the paper is as follows. In section 1 we will prove 
the behaviour of the maximal solution in the inner region. In section 
two we will prove the behaviour of the maximal solution in the outer 
region.

\section{Inner region behaviour}
\setcounter{equation}{0}
\setcounter{thm}{0}

In this section we will prove the behaviour of the maximal solution in 
the inner region near the extinction time using a modification of the
argument of \cite{DP2} and \cite{DS2}.

\begin{lem}
The solution $u$ satisfies
\begin{equation*}
u(x,t)\ge u(y,t)
\end{equation*}
for any $t\in (0,T)$ and $x,y\in\R^2$ such that $|y|\ge |x|+2r_2$.
\end{lem}
\begin{proof}
We will use a modification of the proof of Lemma 2.1 of \cite{CF}
to prove the lemma. 

For any $x^0,y^0\in\R^2$ such that $|y^0|\ge |x^0|+2r_2$ let $\Pi$ be the 
hyperplane of points in $\R^2$ which are equidistance from $x^0$ and $y^0$.
Then (cf. Lemma 2.1 of \cite{CF}),
$$
\Pi=\{x\in\R^2:x\cdot (x^0-y^0)=\mbox{{\small$\frac{1}{2}$}}
(x^0+y^0)\cdot (x^0-y^0)\}
$$
and
\begin{equation}
\mbox{dist}(\Pi,\{0\})=\frac{1}{2}\cdot\frac{|y^0|^2-|x^0|^2}{|x^0-y^0|}
\ge\frac{1}{2}\cdot\frac{|y^0|^2-|x^0|^2}{|x^0|+|y^0|}
\ge\frac{1}{2}(|y^0|-|x^0|)\ge r_2.
\end{equation}
We write $\R^2\setminus\Pi=\Pi_+\cup\Pi_-$ where $\Pi_+$ and $\Pi_-$
are the two half-spaces with respect to $\Pi$ with $0\in\Pi_-$.
By (1.1) $\{x^0\}\cup B_{r_2}\subset\Pi_-$ and $y^0\in\Pi_+$. By 
rotation we may assume without loss of generality that 
$$\left\{\begin{aligned}
&\Pi=\{(x_1,x_2)\in\R^2:x_1=a_0\}\\
&\Pi_-=\{(x_1,x_2)\in\R^2:x_1<a_0\}\\
&\Pi_+=\{(x_1,x_2)\in\R^2:x_1>a_0\}
\end{aligned}\right.
$$
where $a_0=\mbox{dist}(\Pi,\{0\})$. For any $x=(x_1,x_2)\in\Pi_-$ let
$\4{x}=(2a_0-x_1,x_2)$ be the reflection point of $x$ about $\Pi$. 
Then if $x\in\2{B}_{r_1}$, by (0.7) and (1.2)
\begin{equation}
u_0(x)\ge u_0(\4{x}).
\end{equation}  
If $x\in\Pi_-\setminus\2{B}_{r_1}$, then 
\begin{align*}
2a_0-x_1>|x_1|\quad\Rightarrow\quad&|\4{x}|\ge |x|>r_1\\
\Rightarrow\quad&u_0(x)\ge u_0(\4{x})\quad\mbox{(by (0.4))}.
\end{align*}
Hence (1.2) holds for any $x\in\Pi_-$. By the maximum principle for 
the equation
\begin{equation*}
u_t=\Delta\log u
\end{equation*}  
in the half-space $\Pi_-$ (cf. Lemma 2.5 of \cite{ERV2}),
\begin{equation*}
u(x,t)\ge u(\4{x},t)\quad\forall x\in\Pi_-, 0<t<T.
\end{equation*}  
Hence 
\begin{equation*}
u(x^0,t)\ge u(y^0,t)
\end{equation*}
and the lemma follows.
\end{proof}

By Lemma 1.1 for any $0<t<T$ there exists $x_t\in \2{B}_{2r_2}$ such that
\begin{equation}
u(x_t,t)=\max_{x\in\R^2}u(x,t).
\end{equation}
Similar to \cite{DP2} we let
\begin{equation}
\2{u}(x,\tau)=\tau^2u(x,t),\qquad\tau=\frac{1}{T-t},\tau>1/T.
\end{equation}
Then $\2{u}$ satisfies \cite{DP2},
\begin{equation}
\2{u}_{\tau}=\Delta\log\2{u}+\frac{2\2{u}}{\tau}\quad\mbox{ in }\R^2\times
(1/T,\infty).
\end{equation}
Let $R_{max}(t)=\max_{x\in\R^2}R(x,t)$ and let $W(t)$ be the width function
with respect to the metric $g_{ij}(t)$ as defined by P.~Daskalopoulos and 
R.~Hamilton in \cite{DH}. We recall a result of \cite{DH}.

\begin{thm} \cite{DH}
There exist positive constants $c>0$ and $C>0$ such that 
\begin{enumerate}
\item[(i)] $c(T-t)\le W(t)\le C(T-t)$
\item[(ii)] $\frac{c}{(T-t)^2}\le R_{max}(t)\le\frac{C}{(T-t)^2}$
\end{enumerate}
hold for any $0<t<T$.
\end{thm}

Let 
$$
\2{R}(x,\tau)=-\frac{\Delta\log\2{u}}{\2{u}}.
$$
Similar to the argument on P.862--863 of \cite{DP2} by Theorem 1.2,
(1.4) and the Aronson-B\'enlian inequality (0.8),
\begin{align}
&-\frac{1}{\tau^2t}\le\2{R}(x,\tau)\le C\quad\forall (x,\tau)\in
\R^2\times (1/T,\infty), t=T-(1/\tau)\nonumber\\
\Rightarrow\quad&-\frac{2}{\tau^2T}\le\2{R}(x,\tau)\le C\quad\forall 
(x,\tau)\in\R^2\times (2/T,\infty).
\end{align}
By (1.5) and (1.6),
\begin{align}
C\ge&\2{R}(x,\tau)=-\frac{\Delta\log\2{u}}{\2{u}}
=\frac{-\2{u}_{\tau}+\frac{2\2{u}}{\tau}}{\2{u}}
=-\frac{\2{u}_{\tau}}{\2{u}}+\frac{2}{\tau}\ge-\frac{2}{\tau^2T}
\quad\mbox{ in }\R^2\times (2/T,\infty)\\
\Rightarrow\qquad&\frac{2}{\tau}+\frac{2}{\tau^2T}
\ge\frac{\2{u}_{\tau}}{\2{u}}\ge-C+\frac{2}{\tau}
\ge-C\qquad\quad\mbox{ in }\R^2\times (2/T,\infty).
\end{align}

\begin{thm}
For any sequence $\{\tau_k\}_{k=1}^{\infty}$, $\tau_k\to\infty$ as 
$k\to\infty$, let $x_{t_k}\in \2{B}_{2r_2}$ be given by (1.3) with
$t=t_k$ and
\begin{equation}
\2{u}_k(y,\tau)=\alpha_k\2{u}(\alpha_k^{\frac{1}{2}}y+x_{t_k},\tau+\tau_k),
\quad y\in\R^2,\tau>-\tau_k+T^{-1}
\end{equation}
where 
\begin{equation}
t_k=T-\tau_k^{-1}\quad\forall k\in\Z^+
\end{equation}
and 
\begin{equation}
\alpha_k=1/\2{u}(x_{t_k},\tau_k).
\end{equation}
Then $\{\2{u}_k\}_{k=1}^{\infty}$ has a subsequence 
$\{\2{u}_{k_i}\}_{i=1}^{\infty}$ that converges uniformly on $C^{\infty}(K)$
for any compact set $K\subset\R^2\times (-\infty,\infty)$ as $i\to\infty$ to a 
positive solution $U$ of equation
\begin{equation}
U_{\tau}=\Delta\log U\quad\mbox{ in }\R^2\times (-\infty,\infty)
\end{equation}
with uniformly bounded non-negative scalar curvature and uniformly 
bounded width on $\R^2\times (-\infty,\infty)$ with respect to the 
metric $\4{g}_{ij}(t)=U(\cdot,t)\delta_{ij}$.
\end{thm}
\begin{proof}
We first observe that by (1.5),
\begin{equation}
\2{u}_{k,\tau}=\Delta\log\2{u}_k+\frac{2\2{u}_k}{\tau+\tau_k}\quad
\mbox{ in }\R^2\times (-\tau_k+(1/T),\infty)
\end{equation}
with
\begin{equation}
\2{u}_k(0,0)=1\quad\mbox{ and }\quad\2{u}_k(y,0)\le 1\quad\forall 
y\in\R^2.
\end{equation}
Since $(\log\2{u}_k)_{\tau}=(\log\2{u})_{\tau}$, by (1.8),
\begin{equation}
-C\le\frac{\2{u}_{k,\tau}}{\2{u}_k}\le\frac{2}{\tau+\tau_k}
+\frac{2}{(\tau+\tau_k)^2T}\le\frac{3T}{2}\quad\mbox{ in }
\R^2\times (2/T-\tau_k,\infty).
\end{equation}
For any $-\infty<a<0<b<\infty$ we choose $k_0\in\Z^+$ such that
$-\tau_k+(2/T)<a$ for any $k\ge k_0$. Then by (1.14) and (1.15) 
there exists a constant $M_1>0$ such that
\begin{equation}
\2{u}_k(x,\tau)\le M_1\quad x\in\R^2,a\le\tau\le b,k\ge k_0.
\end{equation}
By (1.15) and (1.16) there exists a constant $C>0$ such 
that
\begin{equation}
|\2{u}_{k,\tau}(x,\tau)|\le CM_1\quad x\in\R^2,a\le\tau\le b,k\ge k_0.
\end{equation}
To complete the proof of the theorem we need the following two
technical lemmas.

\begin{lem}
There exists a constant $C_1>0$ such that
\begin{align}
-C_1R_1^2M_1+\log\biggl(\frac{1}{\2{u}_k(x_0,\tau)}\biggr )
\le&\frac{1}{|B_{R_1}|}\int_{B_{R_1}(x_0)}\log
\biggl(\frac{1}{\2{u}_k(x,\tau)}\biggr )\,dx\nonumber\\
\le&C_1R_1^2M_1+\log\biggl(\frac{1}{\2{u}_k(x_0,\tau)}\biggr )
\end{align}
holds for any $R_1>0$, $x_0\in\R^2$, $a\le\tau\le b$ and $k\ge k_0$.
\end{lem}
\begin{proof}
We will use a modification of the proof of Lemma 6 of \cite{V}
and Lemma 2.6 of \cite{Hu1} to proof the lemma. Let
$$
G_{R_1}(x)=\log (R_1/|x-x_0|)+\frac{1}{2}R_1^{-2}
(|x-x_0|^2-R_1^2)
$$
be the Green function for $B_{R_1}(x_0)$. Then $G_{R_1}\ge 0$ and 
$\Delta G_{R_1}=2R_1^{-2}-2\pi\delta_0$ where $\delta_0$ is the delta 
mass at the origin. By (1.13),
\begin{align}
&\int_{B_{R_1}(x_0)}G_{R_1}(x)\2{u}_{k,\tau}(x,\tau)\,dx\nonumber\\
=&\int_{B_{R_1}(x_0)}G_{R_1}(x)\left(\Delta\log\2{u}_k(x,\tau)
+\frac{2\2{u}_k(x,\tau)}{\tau+\tau_k}\right)\,dx\nonumber\\
=&2\pi\left(\log\left(\frac{1}{\2{u}_k(x_0,\tau)}\right)
-\frac{1}{|B_{R_1}|}\int_{B_{R_1}(x_0)}
\log\left(\frac{1}{\2{u}_k(x,\tau)}\right)\,dx\right)\nonumber\\
&\qquad +\frac{2}{\tau+\tau_k}\int_{B_{R_1}(x_0)}G_{R_1}(x)
\2{u}_k(x,\tau)\,dx\qquad\forall\tau\ge-\tau_k+\frac{2}{T},k\ge k_0.
\end{align}
Since 
$$
\int_{B_{R_1}(x_0)}G_{R_1}(x)\,dx\le CR_1^2,
$$
by (1.16), (1.17), and (1.19) we get (1.18) and the lemma follows.
\end{proof}

\begin{lem}
For any $R_1>0$ there exists a constant $C_2>0$ such that
\begin{equation*}
\sup_{\tiny\begin{array}{c}
|y|\le R_1\\
a\le\tau_1\le b\end{array}}\2{u}_k(y,\tau_1)^9
\le C_2\inf_{\tiny\begin{array}{c}
|x|\le R_1\\
a\le\tau_2\le b\end{array}}\2{u}_k(x,\tau_2)\quad\forall k\ge k_0.
\end{equation*}
\end{lem}
\begin{proof}
Let $|x_0|,|y_0|\le R_1$, $\tau_1,\tau_2\in [a,b]$ and $k\ge k_0$. 
Since $B_{R_1}(x_0)\subset B_{3R_1}(y_0)$, by Lemma 1.4,
\begin{align*}
\log\left(\frac{1}{\2{u}_k(x_0,\tau_1)}\right)
\le&\frac{1}{|B_{R_1}|}\int_{B_{R_1}(x_0)}
\log\left(\frac{1}{\2{u}_k(x,\tau_1)}\right)\,dx+C_1M_1R_1^2\\
\le&\frac{9}{|B_{3R_1}|}\int_{B_{3R_1}(y_0)}
\log\left(\frac{1}{\2{u}_k(x,\tau_1)}\right)\,dx+C_1M_1R_1^2\\
\le&9\log\left(\frac{1}{\2{u}_k(y_0,\tau_1)}\right)+C'M_1R_1^2
\end{align*}
for some constants $C_1>0$ and $C'=82C_1$. Hence
\begin{equation}
\frac{1}{\2{u}_k(x_0,\tau_1)}
\le\frac{e^{C'M_1R_1^2}}{\2{u}_k(y_0,\tau_1)^9}\quad
\Rightarrow\quad\2{u}_k(y_0,\tau_1)^9\le e^{C'M_1R_1^2}
\2{u}_k(x_0,\tau_1).
\end{equation}
Now by (1.15) there exists a constant $C>0$ such that
\begin{equation}
\2{u}_k(x,\tau_1')\le C\2{u}_k(x,\tau_2')\quad\forall x\in\R^2,
\tau_1',\tau_2'\in [a,b],k\ge k_0.
\end{equation}
By (1.20) and (1.21),
\begin{equation*}
\2{u}_k(y_0,\tau_1)^9\le C_2\2{u}_k(x_0,\tau_2)
\end{equation*}
holds for some constant $C_2>0$ and the lemma follows.
\end{proof}

We will now continue with the proof of Theorem 1.3. By (1.14) and 
Lemma 1.5 for any $R_1>0$ there exist constants $C_3>0$ and $C_4>0$ 
such that
\begin{equation*}
C_3\le\2{u}_k(x,\tau_1)\le C_4\quad\forall |x|\le R_0,a\le\tau\le b,
k\ge k_0.
\end{equation*}
Hence the equation (1.13) for $\2{u}_k$ is uniformly parabolic on 
$\2{B}_{R_1}\times [a,b]$ for all $k\ge k_0$. By the parabolic Schauder 
estimates \cite{LSU} $\2{u}_k$ are uniformly Holder continuous in 
$C^{2\gamma,1\gamma}(\2{B}_{R_1}\times [a,b])$ for any $\gamma\in\Z^+$. 
By the Ascoli theorem and a diagonalization argument the sequence 
$\{\2{u}_k\}_{k=1}^{\infty}$ has a subsequence
which we may assume without loss of generality to be the sequence itself
that converges uniformly in $C^{\infty}(K)$ as $k\to\infty$ for any compact
set $K\subset\R^2\times (-\infty,\infty)$ to some positive function 
$U$ that satisfies (1.12).

Let
\begin{equation}
\2{R}_k=-\frac{\Delta\log\2{u}_k}{\2{u}_k}.
\end{equation}
Then $\2{R}_k$ converges uniformly on every compact subset of 
$\R^2\times (-\infty,\infty)$ as $k\to\infty$ to the scalar curvature
$\4{R}=-(\Delta\log U)/U$ of the metric $\4{g}_{ij}(\tau)
=U(\cdot,\tau)\delta_{ij}$. Note that by (1.6),
\begin{align*}
&-\frac{2}{(\tau+\tau_k)^2T}\le\2{R}_k(y,\tau)\le C\quad\forall 
(y,\tau)\in\R^2\times (2/T-\tau_k,\infty)\\
\Rightarrow\quad&0\le\4{R}(y,\tau)\le C\quad\forall 
(y,\tau)\in\R^2\times (-\infty,\infty)\quad\mbox{ as }k\to\infty.
\end{align*}
Finally similar to the argument on P.10 of \cite{DS2} by Theorem 1.2
and an approximation argument the width function with respect to the
metric $\4{g}_{ij}(\tau)=U(\cdot,\tau)\delta_{ij}$ is uniformly bounded
on $\R^2\times (-\infty,\infty)$.  
\end{proof}

\begin{thm}
For any sequence $\{\tau_k\}_{k=1}^{\infty}$, $\tau_k\to\infty$ as 
$k\to\infty$, let $\2{u}_k$ be given by (1.9).
Then $\{\2{u}_k\}_{k=1}^{\infty}$ has a subsequence 
$\{\2{u}_{k_i}\}_{i=1}^{\infty}$ that converges uniformly 
in $C^{\infty}(K)$ for any compact set $K\subset \R^2\times
(-\infty,\infty)$ as $i\to\infty$ to 
\begin{equation}
U(y,\tau)=\frac{1}{\lambda|y|^2+e^{4\lambda\tau}}.
\end{equation}
for some constant $\lambda>0$.
\end{thm}
\begin{proof}
By Theorem 1.3 $\{\2{u}_k\}_{k=1}^{\infty}$ 
has a subsequence $\{\2{u}_{k_i}\}_{i=1}^{\infty}$ that converges 
uniformly in $C^{\infty}(K)$ for any compact set $K\subset \R^2\times
(-\infty,\infty)$ as $i\to\infty$ to a solution $U(y,\tau)$ of (1.12). 
By Theorem 1.3 and the result of \cite{DS1}, 
\begin{equation}
U(y,\tau)=\frac{2}{\beta(|y-y_0|^2+\delta e^{2\beta\tau})}
\end{equation} 
for some $y_0\in\R^2$ and constants $\beta>0$, $\delta>0$. Since 
$\2{u}_{k_i}$ converges uniformly on every compact subset of 
$\R^2\times (-\infty,\infty)$ to $U(y,\tau)$ and $\2{u}_k(y,0)$ 
attains its maximum at $y=0$, $U(y,0)$ will attain its maximum at 
$y=0$. Hence $y_0=0$. By (1.14),
\begin{equation}
U(0,0)=1\quad\Rightarrow\quad 1=\frac{2}{\beta\delta}.
\end{equation} 
By (1.24) and (1.25) we get (1.23) with $\lambda=\beta/2>0$.
\end{proof}

We now let 
\begin{equation}
\alpha (\tau)=1/\2{u}(x_t,\tau)
\end{equation} 
where $\tau=1/(T-t)$, $\tau>1/T$, and $x_t\in\2{B}_{2r_2}$ satisfies (1.3).

\begin{lem}
There exist constants $\delta>0$ and $\tau_0>1/T$ such that 
\begin{equation}
\liminf_{\delta'\to 0^+}\biggl(
\frac{\log\alpha(\tau)-\log\alpha (\tau-\delta')}{\delta'}\biggr)>\delta
\quad\forall\tau\ge\tau_0.
\end{equation}
\end{lem}
\begin{proof}
Suppose (1.27) does not hold. Then there exist a sequence of positive numbers
$\{\delta_k\}_{k=1}^{\infty}$, $\delta_k\to 0$ as $k\to\infty$, and a 
sequence $\{\tau_k\}_{k=1}^{\infty}$, $\tau_k>1/T$ for all $k\in\Z^+$
and $\tau_k\to\infty$ as $k\to\infty$, such that 
\begin{equation*}
\liminf_{\delta'\to 0^+}\biggl(
\frac{\log\alpha(\tau_k)-\log\alpha (\tau_k-\delta')}{\delta'}\biggr)
\le\delta_k\quad\forall k\in\Z^+.
\end{equation*}
Hence for each $k\in\Z^+$
there exists a sequence of positive numbers $\{\delta_{k,j}\}_{j=1}^{\infty}$,
$\delta_{k,j}\to 0$ as $j\to\infty$, such that
\begin{equation}
\frac{\log\alpha(\tau_k)-\log\alpha (\tau_k-\delta_{k,j})}{\delta_{k,j}}
<2\delta_k\quad\forall k,j\in\Z^+.
\end{equation}
Let $t_k$ be given by (1.10) and let $t_{k,j}=T-(\tau_k-\delta_{k,j})^{-1}$.
Let $x_{t_k}$ and $x_{t_{k,j}}$ be given by (1.3) with $t=t_k, t_{k,j}$ 
respectively. Then by (1.26) and (1.28),
\begin{equation}
\frac{\log\2{u}(x_{t_{k,j}},\tau_k-\delta_{k,j})-\log\2{u}(x_{t_k},\tau_k)}{\delta_{k,j}}
<2\delta_k\quad\forall k,j\in\Z^+.
\end{equation}
Since
\begin{equation*}
\2{u}(x_{t_k},\tau_k-\delta_{k,j})\le\max_{z\in\R^2}\2{u}
(z,\tau_k-\delta_{k,j})
=\2{u}(x_{t_{k,j}},\tau_k-\delta_{k,j}), 
\end{equation*}
by (1.29),
\begin{align}
&\frac{\log\2{u}(x_{t_k},\tau_k-\delta_{k,j})-\log\2{u}(x_{t_k},\tau_k)}{\delta_{k,j}}
<2\delta_k\quad\forall k,j\in\Z^+\nonumber\\
\Rightarrow\quad&\frac{\2{u}_{\tau}}{\2{u}}(x_{t_k},\tau_k)
\ge -2\delta_k\quad\forall k\in\Z^+\quad\mbox{ as }j\to\infty
\nonumber\\
\Rightarrow\quad&\frac{\Delta\log\2{u}}{\2{u}}(x_{t_k},\tau_k)
+\frac{2}{\tau_k}\ge -2\delta_k\quad\forall k\in\Z^+\qquad 
(\mbox{by (1.5)})\nonumber\\
\Rightarrow\quad&-\2{R}(x_{t_k},\tau_k)
+\frac{2}{\tau_k}\ge -2\delta_k\quad\forall k\in\Z^+.
\end{align}
Let $\2{u}_k$ be given by (1.9) with $\alpha_k=\alpha (\tau_k)$ and
$\2{R}_k$ be given by (1.22). Since $\2{R}_k(0,0)=\2{R}(x_{t_k},\tau_k)$,
by (1.30),
\begin{equation}
\2{R}_k(0,0)\le 2\delta_k+\frac{2}{\tau_k}\quad\forall k\in\Z^+.
\end{equation}
By Theorem 1.3 and Theorem 1.6 $\2{u}_k$ has a subsequence which we may 
assume without loss of generality to be the sequence itself that converges
uniformly on $C^{\infty}(K)$ for any compact set $K\subset\R^2\times 
(-\infty,\infty)$ as $k\to\infty$ to some function $U(y,\tau)$ given by 
(1.23) for some $\lambda>0$. Then
\begin{equation}
\2{R}_k(0,0)\to -\frac{\Delta\log U}{U}(0,0)=4\lambda\quad\mbox{ as }
k\to\infty.
\end{equation}
Letting $k\to\infty$ in (1.31), by (1.32) we get
\begin{equation*}
4\lambda\le 0.
\end{equation*}
Since $\lambda>0$, contradiction arises. Hence there exist constants 
$\delta>0$ and $\tau_0>1/T$ such that (1.27) holds.
\end{proof}

\begin{cor}
Let $\delta>0$ and $\tau_0>1/T$ be as given by Lemma 1.7. Then
\begin{equation*}
\alpha(\tau)\ge\alpha (\tau_0)\,e^{\delta(\tau-\tau_0)}\quad\forall\tau
\ge\tau_0.
\end{equation*}
\end{cor}
\begin{proof}
By Lemma 1.7 there exists constants $\delta>0$ and $\tau_0>1/T$ such that 
(1.27) holds. Let $\tau>\tau_0$. By (1.27) there exists a constant 
$\delta_0'>0$ such that
\begin{equation}
\log\alpha(\tau)-\log\alpha (\tau-\delta')>\delta\delta'
\quad\forall 0<\delta'\le\delta_0'.
\end{equation}
Let 
$$
\delta_0=\sup\{\delta_1>0:\log\alpha(\tau)-\log\alpha (\tau-\delta')
\ge\delta\delta'\quad\forall 0<\delta'\le\delta_1\}.
$$
Then by (1.33) $\delta_0\ge\delta_0'$. We claim that $\delta_0\ge\tau -\tau_0$.
Suppose not. Then $\delta_0<\tau -\tau_0$. By continuity,
\begin{equation}
\log\alpha(\tau)-\log\alpha (\tau-\delta')\ge\delta\delta'
\quad\forall 0<\delta'\le\delta_0.
\end{equation}
Since $\tau -\delta_0>\tau_0$, by (1.27) there exists a constant 
$\delta_1'>0$ such that
\begin{equation}
\log\alpha(\tau-\delta_0)-\log\alpha (\tau-\delta_0-\delta')
>\delta\delta'\quad\forall 0<\delta'\le\delta_1'.
\end{equation}
By (1.34) and (1.35),
\begin{align*}
&\log\alpha(\tau)-\log\alpha (\tau-(\delta_0+\delta'))
>\delta(\delta_0+\delta')\quad\forall 0<\delta'\le\delta_1'\\
\Rightarrow\quad&\log\alpha(\tau)-\log\alpha (\tau-\delta')
\ge\delta\delta'\qquad\qquad\qquad\,\,\,
\forall 0<\delta'\le\delta_0+\delta_1'.
\end{align*}
This contradicts the definition of $\delta_0$. Hence
$\delta_0\ge\tau -\tau_0$. Thus
\begin{align*}
&\log\alpha(\tau)\ge\log\alpha (\tau_0)+\delta(\tau-\tau_0)\quad\forall
\tau\ge\tau_0\\
\Rightarrow\quad&\alpha(\tau)\ge\alpha (\tau_0)e^{\delta(\tau-\tau_0)}
\quad\forall\tau\ge\tau_0
\end{align*}
and the corollary follows.
\end{proof}

\begin{cor}
Let $\alpha_k$ be as in Theorem 1.3. Then $\alpha_k\to\infty$ as $k\to\infty$.
\end{cor}

\begin{lem}
Let $\tau_k$, $\tau_{k_i}$, $\alpha_k$, $U(y,\tau)$ and $\lambda>0$ be
as in Theorem 1.3 and Theorem 1.6. Let
\begin{equation}
q_k(y,\tau)=\alpha_k\2{u}(\alpha_k^{\frac{1}{2}}y,\tau+\tau_k)
\end{equation}
where $\2{u}$ is given by (1.4). Then $q_{k_i}(y,\tau)$ converges uniformly 
in $C^{\infty}(K)$ for every compact set $K\subset\R^2$ to the function
$U(y,\tau)$ as $\tau\to\infty$.
\end{lem}
\begin{proof}
Let $\2{u}_k(y,\tau)$ be given by (1.9) with $x_{t_k}$, $t_k$, given by 
(1.3) and (1.10). Then by Theorem 1.3 and Theorem 1.6 
$\2{u}_k(y,\tau)$ has a subsequence which we may assume without loss
of generality to be the sequence itself that converges uniformly on
$C^{\infty}(K)$ for every compact set $K\subset\R^2$ to the function
$U(y,\tau)$ given by (1.23) for some $\lambda>0$ as $k\to\infty$.

By Corollary 1.9 there exists $k_0\in\Z^+$ such that $\alpha_k\ge 1$ 
for all $k\ge k_0$. Let $K\subset\R^2\times (-\infty,\infty)$ be a 
compact set. Without loss
of generality we may assume that $K=\2{B}_{r_0}\times [\tau_0,\tau_0']$
for some $r_0>0$ and $\tau_0<\tau_0'$. Then for any $|y|\le r_0$ and 
$k\ge k_0$, we have 
$$
|y-\alpha_k^{-\frac{1}{2}}x_{t_k}|\le r_0+2r_2.
$$
Since
\begin{equation*}
q_k(y,\tau)=\2{u}_k(y-\alpha_k^{-\frac{1}{2}}x_{t_k},\tau),
\end{equation*}
by Theorem 1.6,
\begin{equation}
q_k(y,\tau)-U(y-\alpha_k^{-\frac{1}{2}}x_{t_k},\tau)\to 0\quad
\mbox{ uniformly on }C^{\infty}(K)\quad\mbox{ as }k\to\infty. 
\end{equation}
Now 
\begin{align*}
|U(y,\tau)-U(y-\alpha_k^{-\frac{1}{2}}x_{t_k},\tau)|
\le&\frac{\lambda}{(\lambda\rho_0+e^{4\lambda\tau})^2}
||y-\alpha_k^{-\frac{1}{2}}x_{t_k}|^2-|y|^2|\nonumber\\
\le&\lambda e^{-8\lambda\tau_0}\alpha_k^{-\frac{1}{2}}|x_{t_k}|
(2|y|+\alpha_k^{-\frac{1}{2}}|x_{t_k}|)\nonumber\\
\le&4\lambda e^{-8\lambda\tau_0}r_2(r_0+r_2)\alpha_k^{-\frac{1}{2}}
\end{align*}
holds for any $|y|\le r_0$ and $\tau_0\le\tau\le\tau_0'$ where $\rho_0$ is 
some constant between $|y|^2$ and $|y-\alpha_k^{-\frac{1}{2}}x_{t_k}|^2$.
Similarly for any $\gamma_0$, $\gamma_1$, 
$\gamma_2\in\Z^+\cup\{0\}$,
\begin{equation}
|\1_{\tau}^{\gamma_0}\1_{y_1}^{\gamma_1}\1_{y_2}^{\gamma_2}U(y,\tau)
-\1_{\tau}^{\gamma_0}\1_{y_1}^{\gamma_1}\1_{y_2}^{\gamma_2}
U(y-\alpha_k^{-\frac{1}{2}}x_{t_k},\tau)|
\le C\alpha_k^{-\frac{1}{2}}
\end{equation}
holds for some constant $C>0$ and any $|y|\le r_0$, $\tau_0\le\tau
\le\tau_0'$. Since $\alpha_k\to\infty$ as $k\to\infty$ by Corollary 1.9, 
by (1.38)
\begin{equation}
U(y,\tau)-U(y-\alpha_k^{-\frac{1}{2}}x_{t_k},\tau)\to 0\quad
\mbox{ uniformly on }C^{\infty}(K)\quad\mbox{ as }k\to\infty. 
\end{equation}
By (1.37) and (1.39) the theorem follows.
\end{proof}

\begin{thm}
Let $\{\tau_k\}_{k=1}^{\infty}$ be a sequence such that 
$\tau_k>1/T$ for all $k\in\Z^+$ and $\tau_k\to\infty$ as $k\to\infty$.
Let 
\begin{equation}
\beta (\tau)=1/\2{u}(0,\tau),
\end{equation} 
$\beta_k=\beta (\tau_k)$ and
\begin{equation}
\2{q}_k(y,\tau)=\beta_k\2{u}(\beta_k^{\frac{1}{2}}y,\tau+\tau_k)
\end{equation} 
where $\2{u}$ is given by (1.4). Then $\2{q}_k$ has a subsequence 
$\2{q}_{k_i}$ 
that converges uniformly on $C^{\infty}(K)$ for any compact set 
$K\subset\R^2\times (-\infty,\infty)$ to some function $U(y,\tau)$ 
given by (1.23) for some $\lambda>0$ as $\tau\to\infty$.
Moreover $\beta_k\to\infty$ as $k\to\infty$.
\end{thm}
\begin{proof}
Let $q_k(y,\tau)$ be given by (1.36) with $\alpha_k$ given by (1.11).
Let $K\subset\R^2\times (-\infty,\infty)$ be a compact set. As before 
we may assume without loss 
of generality that $K=\2{B}_{r_0}\times [\tau_0,\tau_0']$ for some $r_0>0$ 
and $\tau_0<\tau_0'$. By Lemma 1.10 $q_k(y,\tau)$ has a 
subsequence which we may assume without loss of generality to be the 
sequence itself that converges uniformly on $C^{\infty}(K)$ as 
$k\to\infty$. Then 
\begin{equation}
\frac{\alpha_k}{\beta_k}=q_k(0,0)\to U(0,0)=1\quad\mbox{ as }
k\to\infty.
\end{equation}
Hence there exists $k_0\in\Z^+$ and constants $c_2>c_1>0$ such that
\begin{equation}
c_1\le\frac{\beta_k}{\alpha_k}\le c_2\quad\forall k\ge k_0.
\end{equation}
Now
\begin{equation}
\2{q}_k(y,\tau)=\frac{\beta_k}{\alpha_k}
q_k((\beta_k/\alpha_k)^{\frac{1}{2}}y,\tau).
\end{equation}
Then for any $|y|\le r_0$, by (1.43),
$$
(\beta_k/\alpha_k)^{\frac{1}{2}}|y|\le c_2^{\frac{1}{2}}r_0
\quad\forall k\ge k_0.
$$ 
Hence by Lemma 1.10
\begin{equation}
q_k((\beta_k/\alpha_k)^{\frac{1}{2}}y,\tau)\to 
U((\beta_k/\alpha_k)^{\frac{1}{2}}y,\tau)
\end{equation}
uniformly on $C^{\infty}(K)$ as $k\to\infty$. By (1.23), (1.42), (1.44), 
(1.45) and Corollary 1.9 the theorem follows.
\end{proof}

\begin{lem}
Let $\tau_k$, $\tau_{k_i}$, $\beta_k$, $U(y,\tau)$ and $\lambda>0$ 
be as in Theorem 1.11. Then for any $\3>0$ and $M>0$ 
there exist $n_1\in\Z^+$ and $C>0$ such that
\begin{equation}\left\{\begin{aligned}
&\left|u(x,t_{k_i})-\frac{(T-t_{k_i})^2}{\lambda |x|^2+\beta_{k_i}}\right|
<u(0,t_{k_i})\3\quad\forall |x|\le\beta_{k_i}^{\frac{1}{2}}M,
i\ge n_1\\
&u(0,t_{k_i})\le C(T-t_{k_i})^2\qquad\qquad\qquad\,\,\,\,\forall i\ge n_1.
\end{aligned}\right.
\end{equation}
where $t_k$ is given by (1.10).
\end{lem}
\begin{proof}
Let $\2{q}_k$ be given by (1.41) and $k_0$ be as in the proof of Theorem 1.11. By Theorem 1.11 $\2{q}_k$ has a 
subsequence $\2{q}_{k_i}$ that converges uniformly on $C^{\infty}(K)$ for any
compact $K\subset\R^2\times (-\infty,\infty)$ to some function $U(y,\tau)$ 
given by (1.23) as $k\to\infty$. Then there exists $n_1\in\Z^+$ 
such that $k_i\ge k_0$ for all $i\ge n_1$ and 
\begin{align*}
&\left|\2{q}_{k_i}(y,0)-\frac{1}{\lambda |y|^2+1}\right|<\3
\quad\forall |y|\le M,i\ge n_1\\
\Rightarrow\quad&\left|\beta_{k_i}\tau_{k_i}^2u(x,t_{k_i})
-\frac{1}{\lambda\beta_{k_i}^{-1}|x|^2+1}\right|<\3
\quad\forall |x|\le\beta_{k_i}^{\frac{1}{2}}M,i\ge n_1\\
\Rightarrow\quad&\left|u(x,t_{k_i})-\frac{(T-t_{k_i})^2}{\lambda |x|^2+\beta_{k_i}}\right|<u(0,t_{k_i})\3\quad\forall |x|
\le\beta_{k_i}^{\frac{1}{2}}M,i\ge n_1.
\end{align*}
By Corollary 1.8 and (1.43) there exists a constant $C>0$ such that
$$
\beta_{k_i}>C\quad\forall i\ge n_1
$$
and the lemma follows.
\end{proof}

Since the scalar curvature $R(x,t)$ satisfies
$$
R_t=\Delta_{g(t)}R+R^2\quad\mbox{ in }\R^2\times (0,T)
$$
where $\Delta_{g(t)}=\frac{1}{u}\Delta$ is the Laplace-Beltrami 
operator with respect to the metric $g_{ij}(t)$, by (0.6) and
the maximum principle,
\begin{equation}
R(x,t)\ge-\frac{1}{t+\mu}\quad\mbox{ in }\R^2\times (0,T).
\end{equation} 
Then by an argument similar to the proof of Lemma 3.5 of \cite{DS2}
but with Lemma 2.4, Lemma 3.4, and $k(t)$ in the proof there being replaced by Lemma 1.1, Lemma 1.12, and $1/[2(t+\mu)]$ we have the following lemma.

\begin{lem}
The constant $\lambda$ in Theorem 1.11 satisfies
$$
\lambda\ge\frac{T+\mu}{2}.
$$
\end{lem}

By Lemma 1.13 and an argument similar to the proof of Lemma 3.5 of \cite{DP2}
we have the following lemma.

\begin{lem}
Let $\beta(\tau)$ be given by (1.40). Then
$$
\liminf_{\tau\to\infty}\frac{\beta'(\tau)}{\beta(\tau)}\ge 2(T+\mu).
$$
\end{lem}

\begin{cor}
Let $\beta(\tau)$ be given by (1.40). Then
\begin{equation*}
\beta(\tau)\ge e^{2(T+\mu)\tau+o(\tau)}\quad\mbox{ as }\tau\to\infty.
\end{equation*}
\end{cor}

As in \cite{DS2} we consider the cylindrical change of variables,
$$
v(\zeta,\theta,t)=r^2u(r,\theta,t),\quad\zeta=\log r, r=|x|
$$
and
let
\begin{equation}
\4{v}(\xi,\theta,\tau)=\tau^2 v(\tau\xi,\theta,t),\quad\tau=1/(T-t),
\tau\ge 1/T.
\end{equation}
Then $\4{v}$ satisfies
\begin{equation}
\tau\4{v}_{\tau}=\frac{1}{\tau}(\log\4{v})_{\xi\xi}
+\tau(\log\4{v})_{\theta\theta}+\xi\4{v}_{\xi}+2\4{v}
\quad\mbox{ in }\R\times [0,2\pi]\times (1/T,\infty).
\end{equation}

\begin{lem}
With the same notation as Theorem 1.11 for any $\3>0$ there exists 
$n_1\in\Z^+$ such that
\begin{equation*}
\left|\4{v}(\xi,\theta,\tau_{k_i})
-\frac{e^{2\tau_{k_i}\xi}}{\lambda e^{2\tau_{k_i}\xi}+\beta_{k_i}}
\right|
<\frac{e^{2\tau_{k_i}\xi}}{\beta_{k_i}}\3
\quad\forall\xi\le\frac{\log\beta_{k_i}}{2\tau_{k_i}},
\theta\in [0,2\pi],i\ge n_1.
\end{equation*}
\end{lem}
\begin{proof}
Let $\3>0$. By Lemma 1.12 there exists $n_1\in\Z^+$ such that (1.46)
holds with $M=1$. Since
$$
\4{v}(\xi,\theta,\tau_{k_i})=\tau_{k_i}^2e^{2\tau_{k_i}\xi}
u(e^{\tau_{k_i}\xi},\theta,t_{k_i})
$$ 
where $t_{k_i}$ is given by (1.10) with $k=k_i$, by (1.46),
\begin{equation*}
\left|\4{v}(\xi,\theta,\tau_{k_i})
-\frac{e^{2\tau_{k_i}\xi}}{\lambda e^{2\tau_{k_i}\xi}+\beta_{k_i}}
\right|
<\tau_{k_i}^2e^{2\tau_{k_i}\xi}u(0,t_{k_i})\,\3
=\frac{e^{2\tau_{k_i}\xi}}{\beta_{k_i}}\3
\end{equation*}
holds for any $\xi\le\log\beta_{k_i}/(2\tau_{k_i})$, $\theta\in
[0,2\pi]$ and $i\ge n_1$ and the lemma follows.
\end{proof}

By (0.5), Corollary 1.15, Lemma 1.16, and an argument similar to the proof 
of Proposition 3.7 of \cite{DP2} we have the following result.

\begin{prop}
Let $\beta(\tau)$ be given by (1.40). Then
$$
\lim_{\tau\to\infty}\frac{\log\beta(\tau)}{\tau}=2(T+\mu).
$$
\end{prop}

\begin{prop}
Let $\beta(\tau)$ be given by (1.40). Then
\begin{equation}
\lim_{\tau\to\infty}\frac{\beta'(\tau)}{\beta (\tau)}=2(T+\mu).
\end{equation}
\end{prop}
\begin{proof}
Since 
\begin{equation}
(\log\beta(\tau))_{\tau}=-(\log\2{u}(0,\tau))_{\tau}
=-\frac{\2{u}_{\tau}(0,\tau)}{\2{u}(0,\tau)},
\end{equation}
by (1.8) and (1.51),
\begin{equation}
C\ge\frac{\beta'(\tau)}{\beta(\tau)}\ge-\frac{2}{\tau^2T}-\frac{2}{\tau}
\ge-\frac{3T}{2}\quad\forall\tau\ge 2/T.
\end{equation}
By (1.52) any sequence $\{\tau_k\}_{k=1}^{\infty}$, $\tau_k\to\infty$ 
as $k\to\infty$, will have a subsequence 
$\{\tau_{k_i}\}_{i=1}^{\infty}$ such that the limit
$$
\lim_{i\to\infty}\frac{\beta'(\tau_{k_i})}{\beta(\tau_{k_i})}
$$
exists. By the L'Hospital rule and Proposition 1.17,
$$
\lim_{i\to\infty}\frac{\beta'(\tau_{k_i})}{\beta(\tau_{k_i})}
=2(T+\mu).
$$
Since the sequence $\{\tau_k\}_{k=1}^{\infty}$ is arbitrary,
(1.50) holds.
\end{proof}

\begin{prop}
Let $\tau_k$, $\2{q}_k$, be as given in Theorem 1.11 and let
\begin{equation}
\4{R}_k=-\frac{\Delta\log\2{q}_k}{\2{q}_k}
\end{equation}
Then
$$
\lim_{\tau\to\infty}\4{R}_k(0,0)=2(T+\mu).
$$
\end{prop}
\begin{proof}
By (1.5) $\2{q}_k$ satisfies
\begin{equation}
\2{q}_{k,\tau}=\Delta\log\2{q}_k+\frac{2\2{q}_k}{\tau+\tau_k}\quad
\mbox{ in }\R^2\times (-\tau_k+(1/T),\infty).
\end{equation}
By (1.41), (1.51) and (1.54),
\begin{equation}
\4{R}_k(0,0)=\frac{\beta'(\tau_k)}{\beta(\tau_k)}+\frac{2}{\tau_k}.
\end{equation}
Letting $k\to\infty$ in (1.55), by Proposition 1.18 the corollary follows.
\end{proof}

\begin{cor}
The constant $\lambda$ in Theorem 1.11 is equal to $(T+\mu)/2$.
\end{cor}
\begin{proof}
Let $\tau_k$, $\2{q}_{k}$, $\2{q}_{k_i}$, and $U(y,\tau)$ be as in 
Theorem 1.11 and let $\4{R}_k$ be given by (1.53). Then by Theorem 1.11,
\begin{equation}
\lim_{i\to\infty}\4{R}_{k_i}(0,0)=-\lim_{i\to\infty}
\frac{\Delta\log\2{q}_{k_i}}{\2{q}_{k_i}}(0,0)
=-\frac{\Delta\log U}{U}(0,0)=4\lambda.
\end{equation}
By Proposition 1.19 and (1.56) the corollary follows.
\end{proof}

By Theorem 1.11 and Corollary 1.20 we have the following main theorem
of this section.

\begin{thm}
Let $\beta(\tau)$ be given by (1.40) and let 
$$
\2{q}(y,\tau)=\beta (\tau)\2{u}(\beta(\tau)^{\frac{1}{2}}y,\tau)
$$
where $\2{u}$ is given by (1.4). Then $\2{q}(y,\tau)$ converges uniformly 
on $C^{\infty}(K)$ for any compact set $K\subset\R^2$ to the function 
\begin{equation*}
U_{\mu}(y)=\frac{1}{\frac{(T+\mu)}{2}|y|^2+1}
\end{equation*}
as $\tau\to\infty$.
\end{thm}

Since
$$
(\log\beta(\tau))_{\tau}=\2{R}(0,\tau)-\frac{2}{\tau},
$$
by (1.5), by Proposition 1.18 we have the following result.

\begin{prop}
Let $\beta(\tau)$ be given by (1.40). Then
$$
\lim_{\tau\to\infty}\2{R}(0,\tau)=2(T+\mu).
$$
\end{prop}

By Theorem 1.21 and an argument similar to the proof of Lemma 1.16
we have the following result.

\begin{lem}
Let $\beta(\tau)$ be given by (1.40) and let $\4{v}$ be given by (1.48). 
Then for any $\3>0$ there exists $\tau_0>1/T$ such that
\begin{equation*}
\left|\4{v}(\xi,\theta,\tau)
-\frac{e^{2\tau\xi}}{\frac{T+\mu}{2}e^{2\tau\xi}+\beta (\tau)}
\right|
<\frac{e^{2\tau\xi}}{\beta(\tau)}\3\quad\forall\xi
\le\frac{\log\beta (\tau)}{2\tau},\theta\in [0,2\pi],\tau\ge\tau_0.
\end{equation*}
\end{lem}

By Corollary 1.15, Proposition 1.17 and Lemma 1.23 we get the following 
result.

\begin{cor}
Let $\beta(\tau)$ be given by (1.40) and let $\4{v}$ be given by (1.48). 
Then
$$
\int_{-\infty}^{\xi^-}\int_0^{2\pi}\4{v}(\xi,\theta,\tau)\,d\theta
\,d\xi\to 0\quad\mbox{ as }\tau\to\infty
$$
and
$$
\lim_{\tau\to\infty}\4{v}(\xi,\theta,\tau)=0\quad\mbox{ uniformly on }
(-\infty,\xi^-]\times [0,2\pi]
$$
for any $\xi^-<T+\mu$.
\end{cor}

\section{Outer region behaviour}
\setcounter{equation}{0}
\setcounter{thm}{0}

In this section we will prove the behaviour of the maximal solution in 
the outer region without using the Hamilton-Yau Harnack inequality 
for surfaces \cite{HY}. By (0.5), Proposition 1.17, Lemma 1.23, 
Corollary 1.24 and an argument similar to the proof of Lemma 4.1 of 
\cite{DS2} we have the following lemma.

\begin{lem}
Let $\xi(\tau)=(\log\beta(\tau))/2\tau$ with $\beta (\tau)$ given by
(1.40). Let $\4{v}$ be given by (1.48). Then there exists constants 
$C_1>0$, $C_2>0$, $C_3>0$ and $\tau_0>1/T$ such that the following holds.
\begin{enumerate}
\item[(i)]\quad$\4{v}(\xi,\theta,\tau)\le C_1\quad\forall\xi\in\R,
\theta\in [0,2\pi],\tau\ge 1/T$
\item[(ii)]\quad$\4{v}(\xi,\theta,\tau)\ge\frac{C_2}{\xi^2}\quad
\forall\xi\ge\xi(\tau),
\theta\in [0,2\pi],\tau\ge\tau_0$
\item[(iii)]\quad$\4{v}(\xi,\theta,\tau)\le\frac{C_3}{\xi^2}\quad
\forall\xi>0,
\theta\in [0,2\pi],\tau\ge\tau_0$.
\end{enumerate}
Moreover
\begin{equation*} 
\xi (\tau)=T+\mu+o(1)\quad\mbox{ as }\tau\to\infty.
\end{equation*} 
\end{lem}

\begin{lem}
For any $b>a>T+\mu$, there exist constants $C>0$ and $\tau_1>1/T$
such that
\begin{equation}
\max_{a\le\xi\le b}\left|\int_0^{2\pi}(\log\4{v})_{\xi}
(\xi,\theta,\tau)\,d\theta\right|\le C\quad\forall
\tau\ge\tau_1.
\end{equation}
\end{lem}
\begin{proof}
Let $\delta=(a-(T+\mu))/2$. By direct computation the scalar 
curvature $R$ in polar coordinates satisfies
\begin{align}
&R(e^{\tau\xi},\theta,t)=-\frac{(\log\4{v})_{\xi\xi}
+\tau^2(\log\4{v})_{\theta\theta}}{\4{v}}(\xi,\theta,\tau)
\nonumber\\
\Rightarrow\quad&(\log\4{v})_{\xi\xi}(\xi,\theta,\tau)
+\tau^2(\log\4{v})_{\theta\theta}(\xi,\theta,\tau)
=-R(e^{\tau\xi},\theta,t)\4{v}(\xi,\theta,\tau)
\end{align}
where $\tau=1/(T-t)$. Integrating (2.2) over $(\theta,\xi)\in 
[0,2\pi]\times [\xi_1,\xi_2]$, $T+\mu+\delta\le\xi_1
<\xi_2\le b+1$, by (1.47) and Lemma 2.1,
\begin{align}
\int_0^{2\pi}(\log\4{v})_{\xi}(\xi_2,\theta,\tau)\,d\theta
=&-\int_{\xi_1}^{\xi_2}\int_0^{2\pi}R(e^{\tau\xi},\theta,t)\4{v}
(\xi,\theta,\tau)\,d\theta\,d\xi\nonumber\\
&\qquad +\int_0^{2\pi}(\log\4{v})_{\xi}(\xi_1,\theta,\tau)
\,d\theta\nonumber\\
\le&C_1+\int_0^{2\pi}(\log\4{v})_{\xi}(\xi_1,\theta,\tau)
\,d\theta\quad\forall\tau>1/T
\end{align}
for some constant $C_1>0$.
Let $\tau_0>1/T$ be as given in Lemma 2.1. By Lemma 2.1 there
exists $\tau_1>\tau_0$ such that 
\begin{equation}
\xi(\tau)<T+\mu+\delta\quad\forall\tau\ge\tau_1.
\end{equation}
Integrating (2.3) over $\xi_1\in (a-\delta,a)$, by (2.4) and 
Lemma 2.1,
\begin{align}
\int_0^{2\pi}(\log\4{v})_{\xi}(\xi_2,\theta,\tau)\,d\theta
\le&\frac{1}{\delta}\left(C_1
+\left.\int_0^{2\pi}(\log\4{v})(\xi_1,\theta,\tau)
\,d\theta\right|_{\xi_1=a-\delta}^{\xi_1=a}\right)\nonumber\\
\le&\frac{C'}{\delta}\qquad\forall a\le\xi_2\le b,
\tau\ge\tau_1.
\end{align}
Integrating (2.3) over $\xi_2\in (b,b+1)$, by (2.4) and 
Lemma 2.1,
\begin{align}
\int_0^{2\pi}(\log\4{v})_{\xi}(\xi_1,\theta,\tau)\,d\theta
\ge&-C_1+\left.\int_0^{2\pi}(\log\4{v})(\xi_1,\theta,\tau)
\,d\theta\right|_{\xi_1=b}^{\xi_1=b+1}\nonumber\\
\ge&-C''\qquad\forall a\le\xi_1\le b,\tau\ge\tau_1.
\end{align}
By (2.5) and (2.6) we get (2.1) and the lemma follows.
\end{proof}

We now let
\begin{equation}
w(\xi,\theta,s)=\4{v}(\xi,\theta,\tau)
\end{equation}
with
$$
s=\log\tau=-\log (T-t).
$$
Then as in \cite{DS2} by (1.49),
\begin{equation}
w_s=e^{-s}(\log w)_{\xi\xi}+e^s(\log w)_{\theta\theta}
+\xi w_{\xi}+2w
\quad\mbox{ in }\R\times [0,2\pi]\times (-\log T,\infty).
\end{equation}

\begin{thm}
As $\tau\to\infty$, the function $\4{v}$ given by (1.48) converges 
to the function
\begin{equation*}
V(\xi)=\left\{\begin{aligned}
&\frac{2(T+\mu)}{\xi^2}\quad\forall\xi>T+\mu\\
&0\qquad\qquad\,\,\,\forall\xi<T+\mu.
\end{aligned}\right.
\end{equation*}
Moreover the convergence is uniform on $(-\infty,a]$ for any $a<T+\mu$ 
and on $[\xi_0,\xi_0']$ for any $\xi_0'>\xi_0>T+\mu$.
\end{thm}
\begin{proof}
By Corollary 1.24 we only need to prove the convergence of the function
$\4{v}$ to $2(T+\mu)/\xi^2$ for $\xi>T+\mu$. Let $\tau_0$, $\tau_1$, be
given by Lemma 2.1 and Lemma 2.2 respectively. Let $s_0
=\max(\log\tau_0,\log\tau_1)$ and $\{s_k\}_{k=1}^{\infty}$ be a 
sequence such that $s_k\to\infty$ as $k\to\infty$. Let
\begin{equation}
w_k(\xi,\theta,s)=w(\xi,\theta,s+s_k)\quad\forall\xi\in\R,
0\le\theta\le 2\pi,s\ge-\log T-s_k.
\end{equation}
Then by (2.8),
\begin{equation}
w_{k,s}=e^{-(s+s_k)}(\log w_k)_{\xi\xi}
+e^{s+s_k}(\log w_k)_{\theta\theta}+\xi w_{k,\xi}+2w_k
\end{equation}
in $\R\times [0,2\pi]\times (-\log T-s_k,\infty)$. Since 
$$
\int_{\R^2}u(x,t)\,dx=4\pi (T-t)\quad\forall 0<t<T, 
$$
by (1.48), (2.7) and (2.9),
\begin{equation}
\int_{-\infty}^{\infty}\int_0^{2\pi}w_k(\xi,\theta,s)\,
d\theta\,d\xi=4\pi\quad\forall s>-\log T-s_k,k\in\Z^+.
\end{equation}
Let
\begin{equation}
W_k^b(\eta,s)=\int_{\eta}^b\int_0^{2\pi}w_k(\xi,\theta,s)
\,d\theta\,d\xi\quad\forall b\ge\eta>T+\mu,s>-\log T-s_k
,k\in\Z^+
\end{equation}
and
\begin{equation*}
W_k(\eta,s)=\int_{\eta}^{\infty}\int_0^{2\pi}w_k(\xi,\theta,s)
\,d\theta\,d\xi\quad\forall \eta>T+\mu,s>-\log T-s_k,k\in\Z^+.
\end{equation*}
Then 
\begin{equation}
(W_k^b)_{\eta}(\eta,s)=-\int_0^{2\pi}w_k(\eta,\theta,s)
\,d\theta\quad\forall b\ge\eta>T+\mu,s>-\log T-s_k,k\in\Z^+.
\end{equation}
Hence by (2.11) and Lemma 2.1 there exists a constant $C>0$ such 
that
\begin{equation}
|W_k^b(\eta,s)|,|(W_k^b)_{\eta}(\eta,s)|\le C\quad\forall 
b\ge\eta>T+\mu,s>s_0-s_k,k\in\Z^+.
\end{equation}
Now by (2.10), (2.12) and (2.13),
\begin{align}
(W_k^b)_s(\eta,s)=&e^{-(s+s_k)}\int_{\eta}^b\int_0^{2\pi}(\log w_k)_{\xi\xi}
(\xi,\theta,s)\,d\theta\,d\xi+\int_{\eta}^b\int_0^{2\pi}\xi 
w_{k,\xi}(\xi,\theta,s)\,d\theta\,d\xi\nonumber\\
&\qquad+2W_k^b\nonumber\\
=&e^{-(s+s_k)}\int_0^{2\pi}(\log w_k)_{\xi}
(b,\theta,s)\,d\theta-e^{-(s+s_k)}\int_0^{2\pi}(\log w_k)_{\xi}
(\eta,\theta,s)\,d\theta\nonumber\\
&\qquad +b\int_0^{2\pi}w_k(b,\theta,s)\,d\theta
-\eta\int_0^{2\pi}w_k(\eta,\theta,s)\,d\theta
+W_k^b\nonumber\\
=&e^{-(s+s_k)}\int_0^{2\pi}(\log w_k)_{\xi}
(b,\theta,s)\,d\theta-e^{-(s+s_k)}\int_0^{2\pi}(\log w_k)_{\xi}
(\eta,\theta,s)\,d\theta\nonumber\\
&\qquad +b\int_0^{2\pi}w_k(b,\theta,s)\,d\theta
+\eta(W_k^b)_{\eta}(\eta,s)+W_k^b.
\end{align}
By (2.14), Lemma 2.1, and Lemma 2.2 for any $a\in (T+\mu,b]$ there 
exists a constant $C_{a,b}>0$ such that
\begin{equation}
|(W_k^b)_s(\eta,s)|\le C_{a,b}(1+b^{-1}+\eta)\quad\forall a\le\eta
\le b,s>s_0-s_k,k\in\Z^+.
\end{equation}
By (2.14) and (2.16) the sequence $\{W_k^b\}_{k=1}^{\infty}$ is equi-Holder
continuous on $[a,b]\times [-s',\infty]$ for any $a\in (T+\mu,b]$, $s'>0$, 
and $k$ large such that $s_k-s_0>s'$. 

We choose a sequence $\{b_i\}_{i=1}^{\infty}$ of monotonically increasing 
sequence such that $b_i>T+\mu$ 
for any $i\in\Z^+$ and $b_i\to\infty$ as $i\to\infty$. By the Ascoli 
theorem and a diagonalization argument the sequence 
$\{W_k^{b_1}\}_{k=1}^{\infty}$ has a subsequence 
$\{W_{j_{1,k}}^{b_1}\}_{k=1}^{\infty}$ such that $W_{j_{1,k}}^{b_1}$ 
converges uniformly on every compact subset of $(T+\mu,b_1]\times 
(-\infty,\infty)$ to some function $W^{b_1}\in C((T+\mu,b_1]\times 
(-\infty,\infty))$ as $k\to\infty$ .

Similarly the sequence $\{W_{j_{1,k}}^{b_2}\}_{k=1}^{\infty}$ has 
a subsequence $\{W_{j_{2,k}}^{b_2}\}_{k=1}^{\infty}$ with $j_{1,1}<j_{2,1}$ 
such that $W_{j_{2,k}}^{b_2}$ converges uniformly on every compact 
subset of $(T+\mu,b_2]\times (-\infty,\infty)$ to some function
$W^{b_2}\in C((T+\mu,b_2]\times (-\infty,\infty))$ as $k\to\infty$ .
Repeating the above argument for any $i\ge 2$ the sequence 
$\{W_{j_{i-1,k}}^{b_i}\}_{k=1}^{\infty}$ has a subsequence 
$\{W_{j_{i,k}}^{b_i}\}_{k=1}^{\infty}$ with $j_{i-1,1}<j_{i,1}$ 
such that $W_{j_{i,k}}^{b_i}$ converges uniformly on every compact 
subset of $(T+\mu,b_i]\times (-\infty,\infty)$ to some function 
$W^{b_i}\in C((T+\mu,b_i]\times (-\infty,\infty))$ as $k\to\infty$.

For any $k\in\Z^+$ let $j_k=j_{k,1}$. Then $\{j_k\}_{k=1}^{\infty}$ 
is a subsequence
of $\{j_{i,k}\}_{k=1}^{\infty}$ for any $i\in\Z^+$. Hence for any 
$i\in\Z^+$, the sequence $\{W_{j_k}^{b_i}\}_{k=1}^{\infty}$ 
converges uniformly on every compact subset of $(T+\mu,b_i]\times 
(-\infty,\infty)$ to $W^{b_i}$ as $k\to\infty$. Thus we may assume 
without loss of generality that for any $i\in\Z^+$ 
$\{W_k^{b_i}\}_{k=1}^{\infty}$ converges uniformly on every compact 
subset of $(T+\mu,b_i]\times (-\infty,\infty)$ to $W^{b_i}$ as 
$k\to\infty$. 

Now by Lemma 2.1 for each $s\in\R$, $w_k(\cdot,\cdot,s)$ will have 
a subsequence which we may assume without loss of generality to be 
the sequence itself that converges weakly in $L^{\infty}(K)$ to some 
non-negative function $\2{w}(\cdot,\cdot,s)$ as $k\to\infty$
for any compact set $K\subset [T+\mu,\infty)\times [0,2\pi]$.
Putting $b=b_i$ and letting $k\to\infty$ in (2.12),
\begin{equation}
W^{b_i}(\eta,s)=\int_{\eta}^{b_i}\int_0^{2\pi}\2{w}(\xi,\theta,s)
\,d\theta\,d\xi\quad\forall b\ge\eta\ge T+\mu,s\in\R,i\in\Z^+.
\end{equation}
By (2.17) as $i\to\infty$, $W^{b_i}$ will increase monotonically to the 
function,
\begin{equation}
W(\eta,s):=\int_{\eta}^{\infty}\int_0^{2\pi}\2{w}(\xi,\theta,s)
\,d\theta\,d\xi\quad\forall \eta\ge T+\mu,s\in\R.
\end{equation}
By Lemma 2.1,
\begin{equation}
\2{w}(\xi,\theta,s)\le\frac{C}{\xi^2}\quad\forall\xi>0,\theta
\in [0,2\pi],s\in\R.
\end{equation}
Then by (2.19),
\begin{equation}
0\le W(\eta,s)-W^{b_i}(\eta,s)\le\int_{b_i}^{\infty}\int_0^{2\pi}\2{w}
(\xi,\theta,s)\,d\theta\,d\xi\le\int_{b_i}^{\infty}\int_0^{2\pi}
\frac{C}{\xi^2}\,d\theta\,d\xi\le\frac{C'}{b_i}
\end{equation}
holds for any $s\in\R$ and $i\in\Z^+$.

\noindent $\underline{\text{\bf Claim 1}}$: $W_k$ converges uniformly 
to $W$ on every compact set $K\subset 
(T+\mu,\infty)\times (-\infty,\infty)$ as $k\to\infty$. 

{\ni{\it Proof of Claim 1}:} Let $K$ be a compact subset of 
$(T+\mu,\infty)\times (-\infty,\infty)$. By Lemma 2.1, (2.19) and (2.20),
\begin{align}
\|W_k-W\|_{L^{\infty}(K)}\le&\|W_k-W_k^{b_i}\|_{L^{\infty}(K)}
+\|W_k^{b_i}-W^{b_i}\|_{L^{\infty}(K)}
+\|W^{b_i}-W\|_{L^{\infty}(K)}\nonumber\\
\le&\int_{b_i}^{\infty}\int_0^{2\pi}\frac{C}{\xi^2}\,d\theta\,d\xi
+\|W_k^{b_i}-W^{b_i}\|_{L^{\infty}(K)}+\frac{C'}{b_i}\nonumber\\
\le&\frac{C''}{b_i}+\|W_k^{b_i}-W^{b_i}\|_{L^{\infty}(K)}\qquad\qquad\forall
i\in\Z^+.
\end{align}
Letting $k\to\infty$ in (2.21),
\begin{align*}
&\limsup_{k\to\infty}\|W_k-W\|_{L^{\infty}(K)}
\le\frac{C''}{b_i}\quad\forall i\in\Z^+\\
\Rightarrow\quad&\lim_{k\to\infty}\|W_k-W\|_{L^{\infty}(K)}
=0\quad\mbox{ as }i\to\infty
\end{align*}
and Claim 1 follows. 

By Corollary 1.24, Lemma 2.1, (2.7), (2.9), (2.11) and Claim 1,
\begin{equation}
W((T+\mu)^+,s)=4\pi\quad\forall s\in\R.
\end{equation}
Let $a_2>a_1>T+\mu$ and $s_1'<s_2'$.
We choose $k_0\in\Z^+$ such that $s_1'>s_0-s_k$ for any $k\ge k_0$. 
Then by (2.15) for any $e^{-s}\zeta_0\in [a_1,a_2]$, $b>a_2$, $s_2'
\ge s\ge s_1'$ and $k\ge k_0$,
\begin{align}
&\frac{d}{ds}(e^{-s}W_k^b(e^{-s}\zeta_0,s))\nonumber\\
=&e^{-s}(W_{k,s}^b(e^{-s}\zeta_0,s)-e^{-s}\zeta_0W_{k,\eta}^b
(e^{-s}\zeta_0,s)-W_k^b(e^{-s}\zeta_0,s))\nonumber\\
=&e^{-(2s+s_k)}\int_0^{2\pi}(\log w_k)_{\xi}(b,\theta,s)\,d\theta
-e^{-(2s+s_k)}\int_0^{2\pi}(\log w_k)_{\xi}(e^{-s}\zeta_0,\theta,s)
\,d\theta\nonumber\\
&\qquad +be^{-s}\int_0^{2\pi}w_k(b,\theta,s)\,d\theta
\end{align}
By Lemma 2.1 and Lemma 2.2,
\begin{equation}
\left|\frac{d}{ds}(e^{-s}W_k^b(e^{-s}\zeta_0,s))\right|
\le C_be^{-(2s+s_k)}+C\frac{e^{-s}}{b}
\end{equation}
for some constants $C_b>0$ depending on $b$ and $C>0$. By (2.24) for 
any $s_2'>s_1'>s_0-s_k$, $\zeta_0>e^{s_2'}(T+\mu)$ and $b>e^{-s_1'}
\zeta_0$,
\begin{align}
|e^{-s_1'}W_k^b(e^{-s_1'}\zeta_0,s_1')-e^{-s_2'}W_k^b(e^{-s_2'}
\zeta_0,s_2')|
=&\left|\int_{s_1'}^{s_2'}\frac{d}{ds}(e^{-s}W_k^b(e^{-s}\zeta_0,s))
\,ds\right|\nonumber\\
\le&\max_{s_1'\le s\le s_2'}\left|\frac{d}{ds}(e^{-s}W_k^b
(e^{-s}\zeta_0,s))\right|(s_2'-s_1')\nonumber\\
\le&\left(C_be^{-(2s_1'+s_k)}+C\frac{e^{-s_1'}}{b}\right)(s_2'-s_1')
\end{align}
Putting $b=b_i$ and letting first $k\to\infty$ and then $i\to\infty$ 
in (2.25),
\begin{equation}
e^{-s_1'}W(e^{-s_1'}\zeta_0,s_1')=e^{-s_2'}W(e^{-s_2'}\zeta_0,
s_2')
\end{equation}
holds for any $s_1',s_2'\in\R$ and $\zeta_0>\max(e^{s_1'},e^{s_2'})
(T+\mu)$. 

Let $\eta,\2{\eta}>T+\mu$ and $s\in\R$. Let 
$\zeta_0=e^s\eta$ and choose $\2{s}$ such that $\2{\eta}
=e^{-\2{s}}\zeta_0$. Then $\eta/\2{\eta}=e^{\2{s}-s}$. Hence
by (2.26),
\begin{equation}
\eta W(\eta ,s)=\2{\eta}W(\2{\eta},\2{s})\quad\forall
\eta,\2{\eta}>T+\mu,s,\2{s}\in\R.
\end{equation}
Letting $\2{\eta}\to T+\mu$ in (2.27), by (2.22) we get
\begin{equation}
W(\eta ,s)=\frac{4\pi(T+\mu)}{\eta}\quad\forall
\eta>T+\mu,s\in\R.
\end{equation}
We now fix $s\in\R$. By (2.28) and an argument similar to the proof
on P.588 of \cite{DS2} $w_k(\xi,\theta,s)$ converges to $2(T+\mu)/\xi^2$ in 
$L^p([\xi_0,\infty)\times [0,2\pi])$ as $k\to\infty$ for any $p\ge 1$ 
and $\xi_0>T+\mu$. Hence by passing to a subsequence we may assume 
without loss of generality that
\begin{equation}
w_k(\xi,\theta,s)\to\frac{2(T+\mu)}{\xi^2}\quad\mbox{ a.e. }
(\xi,\theta)\in (T+\mu,\infty)\times [0,2\pi]\quad\mbox{ as }k
\to\infty.
\end{equation}
Let 
$$
Z_k(\xi,s)=\int_0^{2\pi}\log w_k(\xi,\theta,s)\,d\theta.
$$
\noindent $\underline{\text{\bf Claim 2}}$: For any $s\in\R$ the 
sequence $Z_k(\xi,s)$ has a subsequence that converges uniformly 
on $[a,b]$ for any $b>a>T+\mu$ to 
\begin{equation}
Z(\xi,s)=2\pi\log\frac{2(T+\mu)}{\xi^2}.
\end{equation}

{\ni{\it Proof of Claim 2}:}
We will use a modification of the proof of Lemma 4.10 of \cite{DS2} to 
prove the claim. By Lemma 2.1 and Lemma 2.2 for any $b>a>T+\mu$ there 
exist constants $C>0$ and $k_1=k_1(s)\in\Z^+$ such that
\begin{equation}
\left\{\begin{aligned}
&|Z_{k,\xi}(\xi,s)|=\left|\int_0^{2\pi}(\log w_k(\xi,\theta,s))_{\xi}
\,d\theta\right|\le C\quad\forall a\le\xi\le b,k\ge k_1\\
&|Z_k(\xi,s)|\le C\qquad\qquad\qquad\qquad\qquad\qquad
\quad\forall a\le\xi\le b,k\ge k_1.
\end{aligned}\right.
\end{equation}
By (2.31), the Ascoli theorem, and a diagonalization argument 
the sequence $Z_k(\xi,\theta,s)$ has a subsequence that converges uniformly 
on $[a,b]$ for any $b>a>T+\mu$ to some function $Z(\xi,s)$ which is
continuous in $\xi>T+\mu$. Without loss of generality we may assume 
that $Z_k(\xi,s)$ converges uniformly on $[a,b]$ for any $b>a>T+\mu$ 
to $Z(\xi,s)$ as $k\to\infty$. Since
$$
\int_{\eta}^{\eta+A}\int_0^{2\pi}\log w_k(\xi,\theta,s)\,d\theta
\,d\xi\to\int_{\eta}^{\eta+A}Z(\xi,s)\,d\xi\quad\forall\eta>T+\mu,
A>0
\quad\mbox{ as }k\to\infty,
$$
by (2.29), Lemma 2.1, and the Lebesgue dominated convergence theorem,
\begin{equation}
\int_{\eta}^{\eta+A}Z(\xi,s)\,d\xi
=\int_{\eta}^{\eta+A}\int_0^{2\pi}\log\frac{2(T+\mu)}{\xi^2}\,d\theta
\,d\xi
=2\pi\int_{\eta}^{\eta+A}\log\frac{2(T+\mu)}{\xi^2}\,d\xi
\end{equation}
holds for any $\eta>T+\mu$, $A>0$. Dividing both sides of (2.32) by $A$
and letting $A\to 0$ we get (2.30) and Claim 2 follows.

By Lemma 1.1,
\begin{align}
w_k(\eta,\theta,s)=&\4{v}(\eta,\theta,\tau_k(s))
=\tau_k(s)^2v(\tau_k(s)\eta,\theta,\tau_k(s))\nonumber\\
=&\tau_k(s)^2e^{2\tau_k(s)\eta}u(e^{\tau_k(s)\eta},\theta,t_k(s))
\nonumber\\
\le&\tau_k(s)^2e^{2\tau_k(s)\eta}u(e^{\tau_k(s)\eta}-2r_2,\theta,t_k(s))
\nonumber\\
=&\tau_k(s)^2e^{2\tau_k(s)\eta}u(e^{\tau_k(s)\eta}(1-2r_2e^{-\tau_k(s)\eta}),
\theta,t_k(s))\nonumber\\
=&\frac{\tau_k(s)^2}{(1-2r_2e^{-\tau_k(s)\eta})^2}v(\tau_k(s)\eta
+\log (1-2r_2e^{-\tau_k(s)\eta}),\theta,t_k(s))\nonumber\\
=&(1-2r_2e^{-\tau_k(s)\eta})^{-2}\4{v}
(\eta+\tau_k(s)^{-1}\log (1-2r_2e^{-\tau_k(s)\eta}),\theta,\tau_k(s))
\nonumber\\
=&(1-2r_2e^{-\tau_k(s)\eta})^{-2}w_k
(\eta+\tau_k(s)^{-1}\log (1-2r_2e^{-\tau_k(s)\eta}),\theta,s)\nonumber\\
\end{align}
for all $k$ large satisfying $e^{\tau_k(s)\eta}>2r_2$ where
$$
\tau_k(s)=e^{s+s_k}\quad\mbox{ and }t_k(s)=T-\tau_k(s)^{-1}.
$$
By Claim 2, (2.33), and an argument similar to the proof on P.590 of
\cite{DS2}, $w_k(\eta,\theta,s)$ converges uniformly 
on every compact subset of $(T+\mu,\infty)\times [0,2\pi]$ to 
\begin{equation*}
\frac{2(T+\mu)}{\eta^2}
\end{equation*}
as $k\to\infty$. Since the sequence $s_k$ is arbitrary, 
$\4{v}(\eta,\theta,\tau)$ converges unfiormly to $2(T+\mu)/\eta^2$ on 
every compact subset of $(T+\mu,\infty)\times [0,2\pi]$ as 
$\tau\to\infty$ and the theorem follows.
\end{proof}


\begin{thebibliography}{99}

\bibitem[CF]{CF} D.G.~Aronson and L.A.~Caffarelli, {\em The initial trace of
a solution of the porous medium equation}, Transactions A.M.S. 280 
(1983), no. 1, 351--366.  

\bibitem[DH]{DH} P.~Daskalopoulos and R.~Hamilton, {\em Geometric estimates
for the logarithmic fast diffusion equation}, Comm. Anal. Geom. 12
(2004), nos. 1--2, 143--164.

\bibitem[DP1]{DP1} P.~Daskalopoulos and M.A.~del Pino, {\em On a singular 
diffusion equation}, Comm. Anal. Geom. 3 (1995), no. 3, 523--542.

\bibitem[DP2]{DP2} P.~Daskalopoulos and M.A.~del Pino, {\em Type II 
collapsing of maximal solutions to the Ricci flow in $\R^2$}, Ann. Inst. 
H. Poincar\'e Anal. Non Linaire 24 (2007), 851--874. 

\bibitem[DS1]{DS1} P.~Daskalopoulos and N.~Sesum, {\em Eternal solutions to 
the Ricci flow on $\R^2$}, Int. Math. Res. Not. 2006, Art. ID 83610, 20 pp.

\bibitem[DS2]{DS2} P.~Daskalopoulos and N.~Sesum, {\em Type II extinction profile of maximal solutions to the Ricci flow equation}, 
J. Geom. Anal. 20 (2010), no. 3, 565--591.

\bibitem[ERV1]{ERV1} J.R.~Esteban, A.~Rodriguez and J.L.~Vazquez, 
{\em The fast diffusion equation with logarithmic nonlinearity and the evolution of conformal metrics in the plane}, Advances in Differential Equations 1 (1996), no. 1, 21--50.

\bibitem[ERV2]{ERV2} J.R.~Esteban, A.~Rodriguez and J.L.~Vazquez, {\em The maximal solution of the logarithmic fast diffusion equation in two space dimensions}, Advances in Differential Equations 2 (1997), no. 6, 867--894. 

\bibitem[G]{G} P.G.~de Gennes, {\em Wetting: statics and dynamics}, Rev. 
Modern Phys. 57 (1985), no. 3, 827--863.  

\bibitem[HY]{HY} R.~Hamilton and S.T.~Yau, {\em The Harnack estimate for the 
Ricci flow - revisited}, Asian J. Math 1 (1997), no. 3, 418--421.

\bibitem[Hs1]{Hs1} S.Y.~Hsu, {\em Large time behaviour of solutions of the Ricci flow equation on $R^2$}, Pacific J. Math 197 (2001), no. 1, 25--41.

\bibitem[Hs2]{Hs2} S.Y.~Hsu, {\em Asymptotic profile of a singular diffusion 
equation as $t\to\infty$}, Nonlinear Analysis TMA 48 (2002), no. 6, 781--790.  

\bibitem[Hs3]{Hs3} S.Y.~Hsu, {\em Asymptotic behaviour of solutions of the equation $u_t=\Delta\log u$ near the extinction time}, Advances in 
Differential Equations 8 (2003), no. 2, 161--187.

\bibitem[Hs4]{Hs4} S.Y.~Hsu, {\em Behaviour of solutions of a singular diffusion equation near the extinction time}, Nonlinear Analysis TMA 56 
(2004), no. 1, 63--104.

\bibitem[Hu1]{Hu1} K.M.~Hui, {\em Existence of solutions of the equation 
$u_t=\Delta\log u$}, Nonlinear Analysis, TMA 37 (1999), no. 7, 875--914. 

\bibitem[Hu2]{Hu2} K.M.~Hui, {\em Singular limit of solutions of the equation 
$u_t=\Delta (u^m/m)$ as $m\to 0$}, Pacific J. Math. 187 (1999), no. 2, 
297--316.

\bibitem[K]{K} J.R.~King, {\em Self-similar behaviour for the equation of fast 
nonlinear diffusion}, Phil. Trans. Royal Soc. London, Series A 343
(1993), 337--375.

\bibitem[LSU]{LSU} O.A.~Ladyzenskaya, V.A.~Solonnikov, and
N.N.~Uraltceva, {\em Linear and quasilinear equations of
parabolic type}, Transl. Math. Mono. Vol 23, Amer. Math. Soc., 
Providence, R.I., U.S.A., 1968.

\bibitem[V]{V} J.L.~Vazquez, {\em Nonexistence of solutions for nonlinear 
heat equations of fast-diffusion type}, J. Math. Pures Appl. 71
(1992), 503--526.

\bibitem[W1]{W1} L.F.~Wu, {\em A new result for the porous medium equation 
derived from the Ricci flow}, Bull. Amer. Math. Soc. 28 (1993), 90--94.

\bibitem[W2]{W2} L.F.~Wu, {\em The Ricci flow on $R^2$}, Comm. Anal. Geom. 
1 (1993), 439--472.

\end{thebibliography}
\end{document}